\documentclass[11pt]{amsart} \textwidth=14.5cm \oddsidemargin=1cm
\usepackage{amsmath}
\usepackage{amsxtra}
\usepackage{amscd}
\usepackage{amsthm}
\usepackage{amsfonts}
\usepackage{amssymb}
\usepackage{eucal}
\usepackage{pmgraph}
\usepackage[usenames,dvipsnames]{color}
\input prepictex
\input pictex
\input postpictex
\newtheorem{thm}{Theorem}[section]

\newtheorem{lem}[thm]{Lemma}
\newtheorem{prop}[thm]{Proposition}

\theoremstyle{definition}
\newtheorem{definition}[thm]{Definition}

\theoremstyle{remark}
\newtheorem{rem}[thm]{Remark}

\numberwithin{equation}{section}

\begin{document}

%Referring commands:
\newcommand{\thmref}[1]{Theorem~\ref{#1}}
\newcommand{\secref}[1]{Section~\ref{#1}}
\newcommand{\lemref}[1]{Lemma~\ref{#1}}
\newcommand{\propref}[1]{Proposition~\ref{#1}}
\newcommand{\corref}[1]{Corollary~\ref{#1}}
\newcommand{\remref}[1]{Remark~\ref{#1}}
\newcommand{\eqnref}[1]{(\ref{#1})}
\newcommand{\exref}[1]{Example~\ref{#1}}

%Simplified symbols:
\newcommand{\nc}{\newcommand}
\nc{\Z}{{\mathbb Z}}
\nc{\C}{{\mathbb C}}
\nc{\N}{{\mathbb N}}
\nc{\R}{{\mathbb R}}
\nc{\F}{{\mf F}}
\nc{\Q}{\ol{Q}}
\nc{\la}{\lambda}
\nc{\ep}{\epsilon}
\nc{\h}{\mathfrak h}
\nc{\n}{\mf n}
\nc{\G}{{\mathfrak g}}
\nc{\DG}{\widetilde{\mathfrak g}}
\nc{\g}{{\mathfrak g}}
\nc{\SG}{\overline{\mathfrak g}}
\nc{\D}{\mc D} \nc{\Li}{{\mc L}} \nc{\La}{\Lambda} \nc{\is}{{\mathbf
i}} \nc{\V}{\mf V} \nc{\bi}{\bibitem} \nc{\NS}{\mf N}
\nc{\dt}{\mathord{\hbox{${\frac{d}{d t}}$}}} \nc{\E}{\mc E}
\nc{\ba}{\tilde{\pa}} \nc{\half}{\frac{1}{2}} \nc{\mc}{\mathcal}
\nc{\mf}{\mathfrak} \nc{\hf}{\frac{1}{2}}
\nc{\hgl}{\widehat{\mathfrak{gl}}} \nc{\gl}{{\mathfrak{gl}}}
\nc{\hz}{\hf+\Z}
\nc{\dinfty}{{\infty\vert\infty}} \nc{\SLa}{\overline{\Lambda}}
\nc{\SF}{\overline{\mathfrak F}} \nc{\SP}{\overline{\mathcal P}}
\nc{\U}{\mathfrak u} \nc{\SU}{\overline{\mathfrak u}}
\nc{\ov}{\overline}
\nc{\wt}{\widetilde}
\nc{\osp}{\mf{osp}}
\nc{\spo}{\mf{spo}}
\nc{\hosp}{\widehat{\mf{osp}}}
\nc{\hspo}{\widehat{\mf{spo}}}
\nc{\I}{\mathbb{I}}
\nc{\X}{\mathbb{X}}
\nc{\Y}{\mathbb{Y}}
\nc{\hh}{\widehat{\mf{h}}}
\nc{\cc}{{\mathfrak c}}
\nc{\dd}{{\mathfrak d}}
\nc{\aaa}{{\mf A}}
\nc{\xx}{{\mf x}}
\nc{\wty}{\widetilde{\mathbb Y}}
\nc{\ovy}{\overline{\mathbb Y}}

\advance\headheight by 2pt

\title{Super duality and homology of unitarizable modules of Lie algebras}

\author[Po-Yi Huang]{Po-Yi Huang}
%\thanks{$^{1}$Partially supported by NSC-grant of the R.O.C.}
\address{Department of Mathematics, National Cheng Kung University, Tainan,
Taiwan 70101} \email{pyhuang@mail.ncku.edu.tw}
\author[Ngau Lam]{Ngau Lam}
%\thanks{$^{2}$Partially supported by NSC-grant of the R.O.C.}
\address{Department of Mathematics, National Cheng Kung University, Tainan,
Taiwan 70101} \email{nlam@mail.ncku.edu.tw}
\author[Tze-Ming To]{Tze-Ming To}
\address{Department of Mathematics, National Changhua University of Education, Changhua, Taiwan, 500}
\email{matotm@cc.ncue.edu.tw}

\newpage
\begin{abstract}
The $\mf{u}$-homology formulas for unitarizable modules at negative levels over classical Lie algebras of infinite rank of types $\mf{gl}(n)$, $\mf{sp}(2n)$ and $\mf{so}(2n)$ are obtained. As a consequence, we recover the Enright's formulas for three Hermitian symmetric pairs of classical types $(SU(p,q),SU(p)\times SU(q))$, $(Sp(2n),U(n))$ and $(SO^\ast(2n),U(n))$.
\end{abstract}
%\noindent{\bf Key words:} L.

%\vspace{.3cm}

%\noindent{\bf Mathematics Subject Classifications (1991)}: 17B67.

\maketitle

\setcounter{tocdepth}{1}
%\tableofcontents

\section{Introduction}
In analogy to Kostant's $\mf{u}$-cohomology formulas \cite{Ko}, Enright establishes
similar formulas \cite{E} for unitarizable highest weight modules of Hermitian symmetric
pairs in term of certain complicated subsets of the Weyl groups. The argument there is
intricate and involves several equivalences of categories and non-trivial combinatorics
of the Weyl groups. Kostant's formula can be rephrased by saying the Kazhdan-Lusztig
polynomials associated to finite-dimensional module are monomials. The same statement is
true by Enright's formulas for unitarizable highest weight modules. Except for the
resemblance of the formulas, there was no obvious connection between Enright's formula
and Kostant's formula.

However, the modules appearing in the Howe duality at negative levels \cite{W,H1,H2} over
classical Lie algebras of infinite rank are unitarizable modules (cf. \cite{EHW}, see
also \propref{unitary} and \remref{rem:unitary} below) and the character formulas for
these modules can be obtained by applying the involution of the ring of symmetric
functions with infinite variables, which sends the elementary symmetric functions to the
complete symmetric functions, to the characters for the corresponding integrable modules
over the respective Lie algebras (cf. \cite{CK, CKW}). Remarkably, the $\mf{u}$-homology
groups of these modules are also dictated by those of the corresponding integrable
modules \cite{CK, CKW}. Recently, the correspondence between  $\mf{u}$-homology groups of
integrable modules at positive levels and $\mf{u}$-homology groups of unitarizable
modules (at negative levels) over the respective Lie algebras can be elucidated in terms
of the so called super duality \cite{CWZ,CW}, established in \cite{BrS,CL,CLW}. So far there is no explanation of the similarity of these two different
$\mf{u}$-homology groups. Super duality gives a first conceptual explanation of this
similarity \cite[Theorem 4.13]{CLW}.

To the best of our knowledge, there is no other proof of Enright's formulas. In this
paper, we give a proof of Enright's homology formulas for
unitarizable modules by using Kostant's formulas and super duality. The
$\mathfrak{u}$-homology formulas (see \thmref{main} below) for unitarizable  modules over classical Lie algebras of
infinite rank of types $\mf{gl}(n)$, $\mf{sp}(2n)$ and $\mf{so}(2n)$ are obtained by
combinatorial method. The proof involves relating the combinatorial data of Kostant's
formulas for integrable modules over corresponding Lie algebras, that are determined by
the super duality, to the data of the Lie algebras under consideration. By applying the
truncation functors (cf. \cite[Section 3.4]{CLW} to the $\mathfrak{u}$-homology formulas,
see also \secref{hw} below), we recover the Enright's formula for three
Hermitian symmetric pairs of classical types $(SU(p,q),SU(p)\times SU(q))$,
$(Sp(2n),U(n))$ and $(SO^\ast(2n),U(n))$. However, for $\mf{so}(2n)$, our method can only
recover partially Enright's formula for some unitarizable highest weight cases.

The paper is organized as follows. In \secref{P}, we review and set up notations for the
classical Lie algebras of finite and infinite rank. We describe the unitarizable highest
weight modules considered in this paper. Combinatorial description of Weyl groups are
also given in this section. In \secref{hw-data}, we compare the actions of certain
subsets of Weyl groups on certain numerical data associated with the highest weights. In
\secref{sec:main}, homology formulas for unitarizable modules over Lie algebras of
infinite rank are proved. In \secref{application}, Enright's homology formulas are
proved.

We shall use the following notations throughout this article. The symbols $\Z$,
$\N$, and $\Z_+$ stand for the sets of all, positive and
non-negative integers, respectively. We set $\Z^*:=\Z\backslash\{0\}$. For a
partition $\la$, we denote by
$\la'$ the transpose partition of $\la$.  Finally
all vector spaces, algebras, tensor products, et cetera, are over the field of
complex numbers $\C$.

\bigskip
\noindent{\bf Acknowledgments.} The second author is very grateful to Shun-Jen Cheng for numerous discussions and useful suggestions. He also thanks Weiqiang Wang for valuable suggestions. The first and second authors were partially supported by an NSC-grant and thank NCTS/South. The third author thanks NCTS/South for hospitality and support.

\section{Preliminaries}\label{P}
\subsection{Classical Lie algebras of infinite rank}\label{sec:LA}
In this subsection we review and fix notations on Lie algebras of interest in this paper. For details we refer to the references \cite{K, W, CK, CLW}.

\subsubsection{The Lie algebra $\mf{a}_\infty$}
Let $\mathbb{C}^{\infty}$ be the vector space over $\C$ with an ordered basis
$\{\,e_i\,|\,i\in\mathbb{Z}\,\}$ so that an element in ${\rm
End}(\C^\infty)$ may be identified with a matrix $(a_{ij})$
($i,j\in\Z$). Let $E_{ij}$ be the matrix with $1$ at the $i$-th row
and $j$-th column and zero elsewhere. Let $\mathring{{\mf a}}_\infty$ denote the
subalgebra of the Lie algebra
${\rm End}(\C^\infty)$ spanned by $E_{ij}$
with $i,j\in\Z$. Denote by ${\mf a}_\infty:=\mathring{{\mf a}}_\infty\oplus\C K$
the central extension of $\mathring{{\mf a}}_\infty$ by the one-dimensional center
$\C K$ given by the $2$-cocycle
\begin{equation}\label{cocycle}
\tau(A,B):={\rm Tr}([J,A]B),
\end{equation}
where $J=\sum_{i\le 0}E_{ii}$ and ${\rm Tr}(C)$ is the trace of the matrix $C$. Observe that the cocycle $\tau$ is a coboundary.
Indeed, there is embedding $\iota_{\mathring{{\mf a}}}$ from $\mathring{{\mf a}}_\infty$  to ${{\mf a}}_\infty$ defined by $A\in \mathring{{\mf a}}_\infty$ sending to $A+{\rm Tr}(JA)K$ (cf. \cite[Section 2.5]{CLW}). It is clear that  $\iota_{\mathring{{\mf a}}}(\mathring{{\mf a}}_\infty)$ is an ideal of ${{\mf a}}_\infty$ and  ${{\mf a}}_\infty$ is a direct sum of the ideals $\iota_{\mathring{{\mf a}}}(\mathring{{\mf a}}_\infty)$ and $\C K$. Note that $\iota_{\mathring{{\mf a}}}({E}_{ii})={E}_{ii}+K$ (resp. ${E}_{ii}$) for $i\le 0$ (resp. $i\ge 1$).

 The Cartan subalgebra $\sum_{i\in \Z}\C E_{ii}\oplus\C K$ is denoted by $\mf{h}_\mf{a}$. By assigning degree $0$ to the Cartan subalgebra and setting ${\rm
deg}E_{ij}=j-i$, $\mf{a}_\infty$ is equipped with a $\Z$-gradation
$\mf{a}_\infty=\bigoplus_{k\in\Z}(\mf{a}_\infty)_k$. This leads to the
following triangular decomposition:
\begin{equation*}
\mf{a}_\infty=(\mf{a}_\infty)_+\oplus (\mf{a}_\infty)_0\oplus
(\mf{a}_\infty)_-,
\end{equation*}
where $(\mf{a}_\infty)_{\pm}=\bigoplus_{k\in\pm\N}(\mf{a}_\infty)_k$ and
$(\mf{a}_\infty)_0=\mf{h}_\mf{a}$.

 The set of simple
coroots, simple roots and positive roots of $\mf{a}_\infty$ are respectively
\begin{align*}
\Pi_{\mf a}^{\vee}=\{\,& \beta_i^{\vee}:=E_{ii}-E_{i+1,i+1}+\delta_{i0}K\mid i\in \Z \, \}, \\
\Pi_{\mf a}=\{\,& \beta_i:=\epsilon_i-\epsilon_{i+1}\mid i\in \Z\, \}, \\
\Delta_{\mf a}^+ =\{\,& \epsilon_i-\epsilon_j\mid i<j,\, i,j\in\Z\,\},
\end{align*}
where $\epsilon_i\in \h_{\mf a}^*$ is determined by $\langle
\epsilon_i,E_{jj}\rangle=\delta_{ij}$ and
$\langle\epsilon_i,K\rangle=0$.  We also let $\vartheta_{\mf{a}}\in \h_{\mf a}^*$ be defined by
$\langle\vartheta_{\mf{a}},K\rangle=1$ and $\langle\vartheta_{\mf{a}},E_{jj}\rangle=0$, for all $j\in \Z$. Let $\rho_{\mf a}\in \h_{\mf a}^*$ be
determined by $\langle \rho_{\mf a},E_{jj}\rangle=-j$, for all $j\in\Z$,
and $\langle \rho_{\mf a},K\rangle=0$, so that we have $\langle
\rho_{\mf a},{\alpha}^{\vee}_i\rangle=1$, for all $i\in \Z$.

\subsubsection{The Lie algebras $\mathfrak{c}_{\infty}$ and $\mathfrak{d}_{\infty}$}
For
${\g}=\mathfrak{c,d}$, let $\mathring{\g}_{\infty}$ be the
subalgebra of $\mathring{\mf{a}}_\infty$ preserving the following bilinear form
on $\mathbb{C}^{\infty}$:
\begin{equation*}
 (e_i|e_j)=
\begin{cases}
(-1)^i\delta_{i,1-j}, & \text{if ${\g}=\mathfrak{c}$}, \\
 \delta_{i,1-j}, & \text{if ${\g}=\mathfrak{d}$},
\end{cases}\quad i,j\in\Z.
\end{equation*}
Let ${\g}_{\infty}=\mathring{\g}_{\infty}\oplus \mathbb{C}K$ be the
central extension of $\mathring{\g}_{\infty}$ determined by the
restriction of the two-cocycle \eqnref{cocycle}. Then ${\g}_{\infty}$
has a natural $\Z$-gradation and a triangular decomposition induced from
$\mathfrak{a}_{\infty}$ with $({\g}_{\infty})_n={\g}_{\infty}\cap
(\mathfrak{a}_{\infty})_n$, for $n\in\Z$. Similar to the $\mf{a}_\infty$ case, the cocycle is a coboundary.
Indeed, there are embeddings $\iota_{\mathring{\g}}$ from $\mathring{\g}_\infty$  to ${\g}_\infty$ defined by $A\in \mathring{\g}_\infty$ sending to $A+{\rm Tr}(JA)K$ \cite[Section 2.5]{CLW}. It is clear that  $\iota_{\mathring{\g}}(\mathring{\g}_\infty)$ is an ideal of ${\g}_\infty$ and  ${\g}_\infty$ is a direct sum of the ideals $\iota_{\mathring{\g}}(\mathring{\g}_\infty)$ and $\C K$. Note that $\iota_{\mathring{\g}}(\wt{E}_{i})=\wt{E}_{i}-K$ for $i\in \N$ where
$$\wt{E}_i=E_{ii}-E_{1-i,1-i}.$$
Note that $({\g}_{\infty})_0=\sum_{i\in\N}\C\widetilde{E}_i\oplus \C
K$ are Cartan subalgebras, which will be denoted by $\mathfrak{h}_{\g}$. We let $\epsilon_i\in
\h_{\g}^*$ be defined by
$\langle\epsilon_i,\widetilde{E}_j\rangle=\delta_{ij}$ for $i,j\in
\N$ and $\langle\epsilon_i,K\rangle=0$. Then the set of
positive roots of $\mathfrak{c}_\infty$ and $\mathfrak{d}_\infty$ are respectively
\begin{align*}
\Delta_{\mf c}^+ =\{\,& \pm\epsilon_i-\epsilon_j\ ,\ -2\epsilon_i\
(i,j\in\N, i<j) \,\}, \\
\Delta_{\mf d}^+ =\{\,& \pm\epsilon_i-\epsilon_j\   (i,j\in\N, i<j) \,\}.
\end{align*}
Set
\begin{equation*}
\begin{aligned}
&\alpha_0^{\vee}=
\begin{cases}
 -\widetilde{E}_1+ K, & \text{for $\mf{c}_\infty$},\\
-\widetilde{E}_1-\widetilde{E}_2+2K, & \text{for $\mf{d}_\infty$},
\end{cases}  \quad \
\alpha_0 =
\begin{cases}
-2\epsilon_1, & \text{for $\mf{c}_\infty$},\\
-\epsilon_1-\epsilon_2, & \text{for $\mf{d}_\infty$}.
\end{cases}
\end{aligned}
\end{equation*}
The set of simple
coroots and simple roots of ${\g}_\infty$ are respectively
\begin{align*}
\Pi_{\g}^{\vee}=\{\,& \alpha_0^{\vee}, \,
  \alpha_i^{\vee}=\widetilde{E}_i-\widetilde{E}_{i+1} \ (i\in\N)\, \}, \\
\Pi_{\g}=\{\,& \alpha_0, \,
 \alpha_i=\epsilon_i-\epsilon_{i+1} \ (i\in\N)\, \}.
\end{align*}
Let $\vartheta_{\g}\in \h_{\g}^*$ defined by
$\langle\vartheta_{\g},\widetilde{E}_i\rangle=0$ for $i\in\N$ and
$\langle\vartheta_{\g},K\rangle=r$ with $r=1$ (resp. $\hf$) for
${\g}=\mathfrak{c}$ (resp. $\mf{d}$).
We also let $\rho_{\g}\in\h_{\g}^*$ be determined by
\begin{equation*}
\begin{aligned}
&\langle \rho_{\g},\widetilde{E}_{j}\rangle=
\begin{cases}
  -j, & \text{for ${\g}=\mf{c}$},\\
 -j+1, & \text{for ${\g}=\mf{d}$},
\end{cases} j\in\N, \qquad\text{and}\qquad \
\langle \rho_{\g},K\rangle =0.
\end{aligned}
\end{equation*}
We have $\langle \rho_{\g},\alpha_i^{\vee}\rangle=1$ for $i\in \N$ and ${\g}=\mathfrak{c, d}$.

\subsubsection{Levi subalgebras}
For ${\g}=\mathfrak{a, c, d}$, let
$\Delta_{\g}:=\Delta_{\g}^+\cup\Delta_{\g}^-$, where
$\Delta_{\g}^-=-\Delta_{\g}^+$. Then $\Delta_{\g}$ is the set of roots of ${\g}_\infty$. Let
$\Delta_{{\g},c}^\pm:=\Delta_{\g}^\pm\cap\big{(}\sum_{j\not=0}\Z
\alpha_j\big{)}$ and
$\Delta_{{\g},n}^\pm:=\Delta_{\g}^\pm\setminus\Delta_{{\g},c}^\pm$. Denote by
${\g}_\alpha$ the root space corresponding to $\alpha\in\Delta_{\g}$. Set
\begin{equation}\label{parabolic}
\begin{aligned}
&{\mf u}_{\g}^{\pm} := \sum_{\alpha\in \Delta_{{\g},n}^\pm}{\g}_{\alpha}, \quad
{\mf l}_{\g}  := \sum_{\alpha\in \Delta_{{\g},c}^\pm}{\g}_{\alpha}\oplus \h_{\g}.
\end{aligned}
\end{equation}
Then we have ${\g}_\infty=\mathfrak u_{\g}^+\oplus\mathfrak{l}_{\g}\oplus\mathfrak{u}_{\g}^-$.
The Lie algebras $\mathfrak{l}_{\g}$ and ${\g}_\infty$ share the same Cartan
subalgebra $\h_{\g}$. Moreover, $\mathfrak{l}_{\g}$ has a triangular decomposition induced from ${\g}_\infty$. For $\mu\in\h_{\g}^*$, we denote respectively by $L({\g}_\infty, \mu)$ and $L(\mathfrak{l}_{\g},\mu)$
the irreducible highest weight ${\g}_\infty$-module and $\mathfrak{l}_{\g}$-module with
highest weight $\mu$ with respect to the triangular decompositions.

For a root $\alpha\in\Delta_{\g}$, ${\g}=\mf{a,c,d}$, define the reflection $\sigma_\alpha$ by
\begin{equation*}
\sigma_\alpha(\mu):=\mu-\langle\mu,{\alpha}^{\vee}\rangle\alpha,  \quad \mu\in\h_{\g}^*.
\end{equation*}
Here and after, ${\alpha}^{\vee}$ denote the coroot of the root $\alpha$. Let $I_\mf{a}=\Z$ and $I_{\g}=\N$ for ${\g}=\mf{c,d}$. For $j\in I_{\g}\cup\{0\}$, let $\sigma_j=\sigma_{\alpha_j}$. Let ${W}_{\g}$ be the subgroup of ${\rm Aut}(\h_{\g}^*)$
generated by the reflections $\sigma_j$ with $j\in I_{\g}\cup\{0\}$, i.e. ${W}_{\g}$ is the Weyl group of
${\g}_\infty$. For each $w\in {W}_{\g}$, $\ell_{\g}(w)$ denote the length of $w$. We also define
\begin{equation*}
w\circ\mu:=w(\mu+\rho_{\g})-\rho_{\g}, \quad\mu\in\h_{\g}^*,w\in {W}_{\g}.
\end{equation*}
Consider ${W}_{{\g},0}$ the subgroup of ${W}_{\g}$ generated by $\sigma_j$ with
$j\not=0$. Let ${W}_{\g}^0$ denote the set of the
minimal length left coset representatives of
$ {W}_{\g}/{W}_{{\g},0}$ (cf.~\cite{V, L, Ku}). We have ${W}_{\g}={W}_{\g}^0{W}_{{\g},0}$. For $k\in\Z_+$, set
$${W}^0_{{\g},k}:=\{\,w\in {W}_{\g}^0\,\vert\, \ell_{\g} (w)=k\,\}.$$

Finally, for ${\g}=\mathfrak{a, c, d}$, let $(\cdot\vert\cdot)$ be a bilinear form defined on subspace of $\h_{\g}^*$
%generated by elements in $\Delta_{\g}$ and $\vartheta_{\g}$ by
satisfying
\[(\epsilon_i\vert\epsilon_j)=\delta_{ij},\,\, \quad (\vartheta_{\g}\vert\epsilon_i)=(\epsilon_i\vert\vartheta_{\g})=(\vartheta_{\g}\vert\vartheta_{\g})=0\quad \text{for}\,\, \, i,j\in I_{\g}.\]
Recall that $I_\mf{a}=\Z$ and $I_{\g}=\N$ for ${\g}=\mf{c,d}$.

\subsection{Finite dimensional Lie algebras}\label{f.d.LA}
For the rest of the paper, let $\g$ stand for $\mathfrak{a,c,d}$.  We shall fix the following notations:
\[
\overline{\mathfrak{a}}:=\mathfrak{a},\,\,\,
\overline{\mathfrak{c}}:=\mathfrak{d},\,\,\,
\overline{\mathfrak{d}}:=\mathfrak{c}.
\]

\begin{rem}
For $\xx=\mf{c,d}$, let $\g^\xx$ and $\ov{\g}^\xx$ be the Lie algebras defined in \cite[Section 2]{CLW} with $m=0$. Then ${\mf{c}}_\infty={\g}^{\mf{c}}$,  ${\mf{d}}_\infty={\g}^{\mf{d}}$, $\overline{\mf{c}}_\infty\cong\ov{\g}^{\mf{c}}$ and $\overline{\mf{d}}_\infty\cong\ov{\g}^{\mf{d}}$. Note that $K$ send to $-K$ for the isomorphisms $\overline{\mf{c}}_\infty\cong\ov{\g}^{\mf{c}}$ and $\overline{\mf{d}}_\infty\cong\ov{\g}^{\mf{d}}$.
\end{rem}

For $m,n\in\N$, the subalgebra of $\mathring{{\mf a}}_\infty$ spanned by $E_{ij}$ with $1-m\le i, j\le n$,
denoted by $\mf{t}_{m,n}\mf{a}$, is isomorphic to the general linear algebra $\mf{gl}(m+n)$.
The subalgebras $(\mf{t}_{n,n}\mf{a})\cap \mathring{{\mf c}}_\infty$ and $(\mf{t}_{n,n}\mf{a})\cap \mathring{{\mf d}}_\infty$  are isomorphic to the symplectic Lie algebra $\mf{sp}(2n)$ and orthogonal Lie algebra $\mf{so}(2n)$, denoted by $\mf{t}_n\mf{c}$ and $\mf{t}_n\mf{d}$ respectively.
We shall drop the subscript of $\mf{t}$ if there has no ambiguity.

For ${{\ov{\g}}}=\mf{a, c, d}$, the embeddings $\iota_{\mathring{\ov{\g}}}$ restricted to $\mf{t}{\ov{\g}}$ are also denoted by $\iota_{\mathring{\ov{\g}}}$. Let $\Delta_{\mf{t}{\ov{\g}}}^+$ denote the set of
positive roots of $\mf{t}{\ov{\g}}$ with respect to the triangular decomposition induced from ${{\ov{\g}}}_\infty$. We also let $\Delta_{\mf{t}{\ov{\g}}}=\Delta_{\mf{t}{\ov{\g}}}^+\cup-\Delta_{\mf{t}{\ov{\g}}}^+$ and $\Delta_{\mf{t}{\ov{\g}},n}^+=\Delta_{{{\ov{\g}}},n}^+\cap\Delta_{\mf{t}{\ov{\g}}}^+$. Set $\h_{\mf{t}{\ov{\g}}}=\h_{{\ov{\g}}}\cap \mf{t}{\ov{\g}}$, ${\mf u}_{\mf{t}{\ov{\g}}}^{\pm}={\mf u}_{{\ov{\g}}}^{\pm}\cap \mf{t}{\ov{\g}}$ and ${\mf l}_{\mf{t}{\ov{\g}}}={\mf l}_{\ov{\g}}\cap \mf{t}{\ov{\g}}$. Note that $\mf{t}{\ov{\g}}$ and ${\mf l}_{\mf{t}{\ov{\g}}}$ share the same a Cartan subalgebra $\h_{\mf{t}{\ov{\g}}}$. Moreover, $\mathfrak{l}_{\mf{t}{\ov{\g}}}$ has a triangular decomposition induced from ${\mf{t}{\ov{\g}}}$. For $\mu\in\h_{\mf{t}{\ov{\g}}}^*$, we denote respectively by $L(\mf{t}{\ov{\g}}, \mu)$ and $L(\mathfrak{l}_{\mf{t}{\ov{\g}}},\mu)$
the irreducible highest weight ${\mf{t}{\ov{\g}}}$-module and $\mf{l}_{\mf{t}{\ov{\g}}}$-module with
highest weight $\mu$ with respect to the triangular decompositions. For $\mu\in\h_{\mf{t}\ov{\g}}^*$, $L(\mf{l}_{\mf{t}{\ov{\g}}},\mu)$ is extended to an $(\mf{l}_{\mf{t}{\ov{\g}}}+{\mf u}_{\mf{t}{\ov{\g}}}^{+})$-module by letting ${\mf u}_{\mf{t}{\ov{\g}}}^{+}$ act trivially. Let $\mf{p}_{\mf{t}{\ov{\g}}}=\mf{l}_{\mf{t}{\ov{\g}}}+{\mf u}_{\mf{t}{\ov{\g}}}^{+}$. Define as usual the parabolic Verma module with highest weight $\mu$ by
 $$
N({\mf{t}\ov{\g}},\mu)={\rm Ind}_{\mf{p}_{\mf{t}{\ov{\g}}}}^{\mf{t}{\ov{\g}}}L(\mf{l}_{\mf{t}{\ov{\g}}},\mu).
 $$

The space $\h_{\mf{t}{\ov{\g}}}^*$ is spanned by $\epsilon_i$ with $1\le i\le n$ (resp. $1-m\le i\le n-1$) for ${{\ov{\g}}}=\mf{c, d}$ (resp. $\mf{a}$) and therefore $\h_{\mf{t}{\ov{\g}}}^*$ can be regarded as a subspace of $\h_{\ov{\g}}^*$. Note that $\h_{\mf{t}{\ov{\g}}}^*$ is an invariant subspace of $\sigma_i$ for $1\le i\le n$ (resp. $1-m\le i\le n-1$) for ${{\ov{\g}}}=\mf{c}$ or $\mf{d}$ (resp. $\mf{a}$). The restriction of these $\sigma_i$ to $\h_{\mf{t}{\ov{\g}}}^*$ are also denoted by $\sigma_i$. Let ${W}_{\mf{t}{\ov{\g}}}$ be the subgroup of ${\rm Aut}(\h_{\mf{t}{\ov{\g}}}^*)$
generated by these $\sigma_i$s. Then ${W}_{\mf{t}{\ov{\g}}}$ is the Weyl group of
$\mf{t}{\ov{\g}}$. For each $w\in {W}_{\mf{t}{\ov{\g}}}$ we let $\ell_{\mf{t}{\ov{\g}}}(w)$ denote the length of $w$.
Consider ${W}_{\mf{t}{\ov{\g}},0}$ the subgroup of ${W}_{\mf{t}{\ov{\g}}}$ generated by $\sigma_j$ with
$j\not=0$. Let ${W}_{\mf{t}{\ov{\g}}}^0$ denote the set of the
minimal length representatives of the left coset space
$ {W}_{\mf{t}{\ov{\g}}}/{W}_{\mf{t}{\ov{\g}},0}$ (cf.~\cite{L, Ku}). For $k\in\Z_+$, set
${W}^0_{{\mf{t}{\ov{\g}}},k}:=\{\,w\in {W}_{\mf{t}{\ov{\g}}}^0\,\vert\, \ell_{\mf{t}{\ov{\g}}}(w)=k\,\}$. We also define
\begin{equation*}
w\circ\mu:=w(\mu+\rho_{\mf{t}{\ov{\g}}})-\rho_{\mf{t}{\ov{\g}}}, \quad\mu\in\h_{\mf{t}{\ov{\g}}}^*,w\in {W}_{\mf{t}{\ov{\g}}}.
\end{equation*}

Finally, let $\rho_{\mf{t}{\ov{\g}}}$ denote the half sum of the positive roots. Then $\rho_{\mf{t}{\ov{\g}}}(h)=\rho_{\ov{\g}}(h)$ (resp. $\rho_{\mf{a}}(h)+\hf(n-m+1)$) for $h\in \h_{\mf{t}{\ov{\g}}}$ with ${{\ov{\g}}}=\mf{c,d}$ (resp. ${{\ov{\g}}}=\mf{a}$).

\subsection{Combinatorial descriptions of Weyl groups}
In this section, we present combinatorial descriptions of certain
aspects of infinite Weyl groups $W_\g$ (cf.~\cite{BB}).
Recall that $\Z^*:=\Z\backslash\{0\}$.

Define $\phi_\g\in\h_{\g}^*$ by
\begin{equation*}
\phi_{\g}=
\begin{cases}
-\sum_{i\le 0}\epsilon_i,& \text{if} \,\,{\g}=\mathfrak{a}; \\
\sum_{i\in \N}\epsilon_i, & \text{if} \,\,{\g}=\mathfrak{c,d}.
\end{cases}
\end{equation*}
Every element $\sigma\in \h_\g^*$ can be uniquely represented by $\sum_{i\in I_\g} \xi_i\epsilon_i+q\vartheta_\g$ with $\xi_i, q\in \C$.  For $\g=\mf{c}, \mf{d}$ and $i\in\N$, we define $\epsilon_{-i}=-\epsilon_i$. It is easy to see by computing the actions of $\sigma_i$ that the actions of $W_\g$ on $\h_\g^*$ is given by
\begin{align}
&\sigma(\sum_{i\in \Z} \xi_i\epsilon_i+q\vartheta_\mf{a})=\sum_{i\le 0} (\xi_i+q)\epsilon_{\tilde{\sigma}(i)}+\sum_{i>  0} \xi_i\epsilon_{\tilde{\sigma}(i)}+q\phi_{\mf{a}}+q\vartheta_\mf{a},  &\text{if}&\,\, \g=\mf{a}
;\label{action-a}\\
&\sigma(\sum_{i\in \N} \xi_i\epsilon_i+q\vartheta_\g)=\sum_{i\in \N} (\xi_i-q\langle\vartheta_\g,K\rangle)\epsilon_{\tilde{\sigma}(i)} +q\langle\vartheta_\g,K\rangle\phi_{\g}+q\vartheta_\g,  &\text{if}&\,\, \g=\mf{c}, \mf{d},\label{action-bcd}
\end{align}
where $\tilde{\sigma}$ is a permutation of $\Z$ (i.e. $\tilde{\sigma}$ is a bijection on $\Z$ satisfying $\tilde{\sigma}(j)=j$ for $|j|\gg 0$) for $\g=\mf{a}$ and $\tilde{\sigma}$ is a signed permutation of $\Z^*$ (i.e. $\tilde{\sigma}$ is a bijection on $\Z^*$ satisfying  $\tilde{\sigma}(j)=j$ for $|j|\gg 0$ and $\tilde{\sigma}(-i)=-\tilde{\sigma}(i)$ for $i\in \Z^*$) for $\g=\mf{c}, \mf{d}$. Therefore ${\sigma}\mapsto\tilde{\sigma}$
is a representation on $\Z$ and $\Z^*$ for $\g=\mf{a}$ and $\g=\mf{c}, \mf{d}$, respectively. Moreover, they are faithful representations. It is clear that the image of $W_{\mf{a}}$ in ${\rm Aut}(\Z)$ is the set of permutations of $\Z$ and the image of $W_{\mf{c}}$ (resp. $W_{\mf{d}}$) in ${\rm Aut}(\Z^*)$ is the set of a signed (resp. even signed) permutations of $\Z^*$. A signed permutation $\tilde{\sigma}$ of $\Z^*$ is called even signed permutation if $|\{i\in\N\mid \tilde{\sigma}(i)<0\}|$ is a even number. We shall identify $W_\g$ with the image of $W_\g$ in ${\rm Aut}(\Z)$ (resp. ${\rm Aut}(\Z^*)$) for $\g=\mf{a}$ (resp. $\mf{c,d}$) for the rest of the paper. Note that for $i\in \Z$, ${\tilde{\sigma}}_i(i)=i+1$, ${\tilde{\sigma}}_i(i+1)=i$ and ${\tilde{\sigma}}_i(j)=j$ for all $j\not= i, i+1$. Also for $\g=\mf{c}, \mf{d}$ and $i\in \N$,  ${\tilde{\sigma}}_i(i)=i+1$, ${\tilde{\sigma}}_i(i+1)=i$ and ${\tilde{\sigma}}_i(j)=j$ for all $j\not= i, i+1$ while ${\tilde{\sigma}}_0(1)=-1$ (resp. $-2$), ${\tilde{\sigma}}_0(2)=2$ (resp. $-1$), and ${\tilde{\sigma}}_0(j)=j$ for all $j\ge 3$ for $\g=\mf{c}$ (resp. $\mf{d}$). We shall use these representations for the rest of the paper and we shall simply write ${\sigma}(j)$ instead of $\tilde{\sigma}(j)$.

Recall that $\ell_\g$ denote the length function on $W_\g$ and ${W}_{\g}^0$ denote the set of the minimal length left coset representatives of $ {W}_{\g}/{W}_{\g,0}$. We have

\begin{align}\label{coset}
W_\g^0=
\begin{cases}\{\sigma\in W_\mf{a}\,|\,\sigma (i)<\sigma (j)\,\, {\rm for}\,\, i< j\le 0\,\, {\rm and}\,\,
\,\, 0< i<j \},&\text{if}\,\, \g=\mf{a};\\
\{\sigma\in W_\g|\sigma (i)<\sigma (j),\, {\rm for}\,\, 1\le i< j\}, &\text{if}\,\, \g=\mf{c,d}
\end{cases}
\end{align}
(see, e.g. \cite[Lemma 2.4.7, Proposition 8.1.4 and Proposition 8.2.4]{BB}) and for $\sigma\in W_\g^0$,
\begin{align}\label{length}
\ell_\g(\sigma)=
\begin{cases}|\{(i,j)\in\Z\times\Z\,\mid\,i<j, \sigma (i)>\sigma (j)\}|,&\text{if}\,\, \g=\mf{a};\\
|\{(i,j)\in\N\times\N\,\mid\, i\leq j,\sigma(-i)>\sigma (j)\}|, &\text{if}\,\, \g=\mf{c};\\
|\{(i,j)\in\N\times\N\,\mid\, i< j,\sigma(-i)>\sigma (j)\}|, &\text{if}\,\, \g=\mf{d}
\end{cases}
\end{align}
(see, e.g. \cite[Corollary 1.5.2, Corollary 8.1.1 and Corollary 8.2.1]{BB}).

\begin{lem}\label{sigma-bar}
For $\sigma\in W_\mf{c}^0$ with $\sigma(i)<0$ for $i\le j$, and $\sigma(i)>0$ for $i> j$, define $\ov{\sigma}\in W_\mf{d}^0$ by
\begin{equation*}
\ov{\sigma}(i)=
\begin{cases}
\sigma(i)-1,& \text{if  } \,\,i\le j; \\
1,& \text{if  } \,\,i=j+1\,\, \text{and $j$ is even}; \\
-1,& \text{if  } \,\,i=j+1\,\, \text{and $j$ is odd}; \\
\sigma(i-1)+1,& \text{if  } \,\,i\ge j+2.
\end{cases}
\end{equation*}
For each $k\ge 0$, the map from $W_{\mf{c},k}^0$ to $W_{\mf{d},k}^0$ sending $\sigma$ to $\ov{\sigma}$ is a bijection.
\end{lem}

\begin{proof} By \eqnref{coset}, it is a bijection from $W_{\mf{c}}^0$ to $W_{\mf{d}}^0$.
By \eqnref{length}, we have $\ell_\mf{c}(\sigma)=\ell_\mf{d}(\ov{\sigma})$ for $\sigma\in W_\mf{c}^0$. The lemma follows.
\end{proof}

Let $\{\xi_i\}_{i\in\N}$ be a sequence of real numbers. Define $\xi_{-i}:=-\xi_i$ for $i\in \N$.
For any sequence of strictly decreasing negative real numbers $\{\xi_i\}_{i\in\N}$ and $\sigma\in W_\g^0$ with $\g=\mf{c, d}$,
it is easy to see that $\{\xi_{\sigma(i)}\}_{i\in\N}$ is a sequence of strictly decreasing real numbers. The following lemma follows from the definition of $\ov{\sigma}$.

\begin{lem}\label{c-d}
Let $\{\xi_i\}_{i\in\N}$ be a sequence of strictly decreasing negative  real numbers. Define
$ \ov{\xi}_{i+1}=\xi_i$ for $i\in\N$ and $\ov{\xi}_1=0$. Then for all $\sigma\in W_\mf{c}^0$, we have
 \[
\{\xi_{\sigma(i)}\,\mid\,i\in\N\}\cup\{ 0\} =\{\ov{\xi}_{\ov{\sigma}(i)}\,\mid\,i\in\N\},
 \]
 where $\ov{\sigma}$ is defined in \lemref{sigma-bar}.
\end{lem}

\subsection{Unitarizable highest weight modules}\label{hw}
Recall that $\g$ stand for $\mathfrak{a,c,d}$, and $\overline{\mathfrak{a}}=\mathfrak{a}$,
$\overline{\mathfrak{c}}=\mathfrak{d}$ and $\overline{\mathfrak{d}}=\mathfrak{c}$.
In this subsection we classify the highest weights of irreducible unitarizable quasi-finite highest weight $\ov\g_\infty$-modules with respect to the anti-linear anti-involution $\omega$ defined below.

For a partition $\lambda=(\lambda_1,\lambda_2,\cdots)$, the
transpose partition of $\lambda$ is denoted by $\lambda'=(\lambda'_1,
\lambda'_2,\cdots)$. For $\mathfrak{g=c, d}$, a
partition $\lambda$ and $d\in\C$, define
\begin{eqnarray}
\quad \Lambda^\g(\lambda,d):=\sum_{i\in\N}\lambda^\prime_i\epsilon_i+d\vartheta_\g\in
\h_\g^\ast,\quad
\SLa^\g(\lambda,d)=\sum_{i\in\N}\lambda_i\epsilon_i -\frac{d\langle\vartheta_\g,K\rangle}{\langle\vartheta_{\ov{\g}},K\rangle}\vartheta_{\ov{\g}}\in \h_{\ov{\g}}^\ast.
\end{eqnarray}
Let $\mc{D}(\g)$ denote the set of pairs $(\la, d)$ with $d\in\Z_+$ satisfying $\la'_1\le d$ if $\g=\mf{c}$; and $\la'_1+\la'_2\le d$ if $\g=\mf{d}$.
For a pair of partitions $\la=(\la^-,\la^+)$ and $d\in\C$,
define $\Lambda^\mathfrak{a}(\lambda,d), \SLa^\mathfrak{a}(\lambda,d)\in\mathfrak{h}_\mathfrak{a}^\ast$ by
\begin{eqnarray*}
&& \Lambda^\mathfrak{a}(\lambda,d)=-\sum_{i\in\Z_+}(\la^-)'_{i+1}\epsilon_{-i} + \sum_{i\in\N}(\la^+)'_{i}\epsilon_i + d\vartheta_\mf{a},
\\&&
\SLa^\mathfrak{a}(\lambda,d)=-\sum_{i\in\Z_+}\la^-_{i+1}\epsilon_{-i} + \sum_{i\in\N}\la^+_{i}\epsilon_i - d\vartheta_\mf{a}.
\end{eqnarray*}
Let $\mc{D}(\mf{a})$ denote the set of pairs $(\la, d)$ satisfying $d\in\Z_+$ and  $(\la^-)'_1+(\la^+)'_1\le d$.

Let $\mf{k}$ be a Lie algebra equipped with an anti-linear anti-involution $\omega$, and
let $V$ be a $\mf{k}$-module. A Hermitian form $\langle\
\cdot\ | \ \cdot \ \rangle$ on $V$ is said to be contravariant if
$\langle a v | v'\rangle = \langle v | \omega(a) v'\rangle$, for
all $a\in \mf{k}$, $v, v'\in V$. A $\mf{k}$-module equipped with a positive
definite contravariant Hermitian form is called a unitarizable
$\mf{k}$-module. Assume that $\mf{k} = \oplus_{j \in \Z} \mf{k}_j$ (possibly
$\dim{\mf{k}_j}=\infty$) is a $\Z$-graded Lie algebra and $\mf{k}_0$ is
abelian. A graded $\mf{k}$-module $M = \oplus_{j \in \Z} M_j $
is called quasi-finite if $ \dim M_j < \infty$ for all
$j\in\Z$ \cite{KR}.

\begin{rem}\label{integrable} Let $V$ be a highest weight $\g_\infty$-module with highest weight $\xi$. Using the arguments as in \cite[Section 4]{LZ}, we have $V$ is quasi-finite if and only if $\xi$ satisfies $\xi(E_{ii})=0$ (resp. $\xi(\widetilde{E}_{ii})=0$) for $|i|\gg 0$ (resp. $i\gg 0$) for $\g=\mf{a}$ (resp. $\mf{c,d}$). Therefore every quasi-finite integrable highest weight $\g_\infty$-module is of the form  $L(\g_\infty,{\Lambda^\g(\la,d)})$ for some $(\la,d)\in\mc{D}(\g)$.
\end{rem}

Now we consider the anti-linear anti-involution $\omega$ on $\mf{a}_\infty$ defined by
\begin{equation*}
\begin{aligned}
\omega(E_{ij})=
\begin{cases}
E_{ji}, & \text{for $i,j \le 0$ or $i,j > 0$}; \\
 -E_{ji}, & \text{for $i>0,j \le 0$ or $i\le 0,j > 0$},
\end{cases}\,\,\, \text{and} \quad
\omega(K)=K.
\end{aligned}
\end{equation*}
For $\mf{g}=\mf{c,d}$, the
restriction of the anti-linear anti-involution $\omega$ on $\mf{a}_\infty$ to $\g_\infty$ gives an anti-linear
anti-involution on $\g_\infty$,  which will also be denoted by $\omega$.

For $d\in\C$ and a pair of partitions $\la=(\la^-,\la^+)$ with $\la_{n+1}^+=\la_{m+1}^-=0$, let $\Gamma_{\mf{t}\ov{\mf{a}}}(\la,d)$ be the element in $\h_{\mf{t}\ov{\mf{a}}}^*$ determined by
\begin{equation*}
\Gamma_{\mf{t}\ov{\mf{a}}}(\la,d)
=\sum_{i=1}^m(-d-\la^-_i)\epsilon_{-i+1} +\sum_{i=1}^n\la^+_i\epsilon_i.
\end{equation*}
For $d\in\C$ and a partition $\la$ satisfying $\la_{n+1}=0$, let $\Gamma_{\mf{t}\ov{\g}}(\la,d)$ be the element in $h_{\mf{t}\ov{\g}}^*$ determined by
\begin{equation*}
\begin{aligned}
\Gamma_{\mf{t}\ov{\g}}(\la,d)
=\begin{cases}\sum_{i=1}^n(\la_i+\frac{d}{2})\epsilon_i,
&\quad \text{for } \ov{\g}=\mf{c},\\
\sum_{i=1}^n(\la_i+d)\epsilon_i,
&\quad \text{for } \ov{\g}=\mf{d}.
\end{cases}\end{aligned}
\end{equation*}
Let $\mc{D}_\mf{t}(\g)$ denote the subset of $\mc{D}(\g)$ consisting of elements in $(\la,d)$ satisfying $\la_{n+1}=0$  for ${{\ov{\g}}}=\mf{c, d}$ (resp. $\la_{n+1}^+=0$ and $\la_{m+1}^-=0$ for ${{\ov{\g}}}=\mf{a}$).

Now we introduce the truncation functors \cite[Section 3.4]{CLW}. Let $M=\bigoplus_{\beta}M_\beta$ be a semisimple $\h_{\ov\g}$-module such that $M_\beta$ is the weight space of $M$ with weight $\beta$. The truncation functor $\mf{tr}_{\mf{t}\ov\h}$ is defined by sending $M$ to $\bigoplus_{\nu}M_\nu$, summed over
$\sum_{i=1-m}^{n}\C \epsilon_i+ \C\vartheta_{\ov\g}$ (resp. $\sum_{i=1}^{n}\C \epsilon_i+ \C\vartheta_{\ov\g}$) for $\ov{\g}=\mf{a}$ (resp. $\mf{c,d}$). For $(\la,d)\in \mc{D}(\g)$,  $L(\ov\g_\infty,\SLa^\g(\la,d))$ is a ${\mf{t}\ov{\g}}$-module through the embedding $\iota_{\mathring{\ov{\g}}}$ defined in \secref{f.d.LA}. $\mf{tr}_{\mf{t}\ov\h}(L(\ov\g_\infty,\SLa^\g(\la,d)))$ is an irreducible ${\mf{t}\ov{\g}}$-module and \begin{equation}\label{trf}
\mf{tr}_{\mf{t}\ov\h}(L(\ov\g_\infty,\SLa^\g(\la,d))) =L({\mf{t}\ov{\g}},\Gamma_{\mf{t}\ov{\g}}(\la,d))
\end{equation}
for any partition $\la$ with $\la_{n+1}=0$ and $\ov{\g}=\mf{c,d}$ \cite[Lemma 3.2]{CLW}. The same result is also true for $\ov{\g}=\mf{a}$ and pair of partitions $\la=(\la^-,\la^+)$ with $\la_{n+1}^+=\la_{m+1}^-=0$ by using the same arguments as in \cite{CLW}. The anti-linear anti-involution $\omega$ on $\ov{\g}_\infty$ induces an anti-linear
anti-involution on $\mf{t}\ov{\g}$,  which will also be denoted by $\omega$.

By cumbersome but straight forward computations, the following theorem is reformulated the Theorem 2.4 and some results of sections $7,8,9$ in \cite{EHW} in terms of partitions.

\begin{thm}\label{EHW} For ${\mf{g}}=\mf{a, c,d}$, let $\xi\in\h_{\mf{t}\ov{\g}}^*$.
\begin{itemize}
\item[i.]   $L({\mf{t}\ov{\mf{a}}},\xi)$ is
unitarizable with respect to $\omega$  if and
only if $\xi=\Gamma_{\mf{t}\ov{\mf{a}}}(\la,d)+k\sum_{i=-m+1}^{n}\epsilon_i$ for some pair of partitions $\la=(\la^+,\la^-)$ with $\la^-_m=\la^+_n=0$ and $d, k\in \R$ satisfying $d\ge {\rm min}\{(\la^-)'_1+n-1, (\la^+)'_1+m-1\}$, or $d\in \Z$ and $d\ge (\la^-)'_1+(\la^+)'_1$. Moreover, $N({\mf{t}\ov{\mf{a}}},\Gamma_{\mf{t}\ov{\mf{a}}}(\la,d)+k\sum_{i=-m+1}^{n}\epsilon_i)$ are irreducible for pair of partitions $\la=(\la^+,\la^-)$ with $\la^-_m=\la^+_n=0$ and $d, k\in \R$ satisfying $d> {\rm min}\{(\la^-)'_1+n-1, (\la^+)'_1+m-1\}$.
\item[ii.] $L({\mf{t}\ov{\mf{d}}},\xi)$ is
unitarizable with respect to $\omega$  if and
only if $\xi=\Gamma_{\mf{t}\ov{\mf{d}}}(\la,d)$ for some partition $\la$ with $\la_n=0$ and $d\in \R$ satisfying $d\ge n-1+\la'_2$, or $d\in \Z$ and $d\ge \la'_1+\la'_2$. Moreover, $N({\mf{t}\ov{\mf{d}}},\Gamma_{\mf{t}\ov{\mf{d}}}(\la,d))$ are irreducible for partition $\la$  with $\la_n=0$ and $d>n-1+\la'_2$.
\item[iii.] Assume that $\xi\in\h_{\mf{t}\ov{\mf{c}}}^*$ with $\xi(\wt{E}_{n-1})=\xi(\wt{E}_n)$. $L({\mf{t}\ov{\mf{c}}},\xi)$ is
unitarizable with respect to $\omega$  if and
only if $\xi=\Gamma_{\mf{t}\ov{\mf{c}}}(\la,d)$ for some partition $\la$  with $\la_{n-1}=\la_n=0$ and $d\in \R$ satisfying $d\ge \hf(\la'_1+n)-1$ if $n-\la'_1$ is even; $d\ge \hf(\la'_1+n-1)-1$ if $n-\la'_1$ is odd, or $d\in \Z$ and $d\ge \la'_1$. Moreover, $N({\mf{t}\ov{\mf{c}}},\Gamma_{\mf{t}\ov{\mf{c}}}(\la,d))$ are irreducible for partition $\la$  with $\la_{n-1}=\la_n=0$ and $d\in \R$ satisfying $d> \hf(\la'_1+n)-1$ if $n-\la'_1$ is even; $d> \hf(\la'_1+n-1)-1$ if $n-\la'_1$ is odd.
\end{itemize}
\end{thm}

\begin{prop}\label{unitary} For ${\mf{g}}=\mf{a, c,d}$, let $L(\ov{\mf{g}}_\infty,\xi)$ be an irreducible quasi-finite highest weight
$\ov{\g}_\infty$-module with highest weight $\xi$. Then $L(\ov{\mf{g}}_\infty,\xi)$ is
unitarizable with respect to the anti-linear anti-involution $\omega$  if and
only if $\xi=\SLa^\g(\la, d)$ for some $(\la, d)\in \mc{D}(\g)$.
\end{prop}

\begin{proof}
Let $L(\ov{\mf{g}}_\infty,\xi)$ be a unitarizable irreducible quasi-finite highest weight
$\ov{\g}_\infty$-module. By \remref{integrable}, $\xi$ satisfies $\xi(E_{ii})=0$ (resp. $\xi(\wt{E}_{ii})=0$) for $|i|\gg 0$ (resp. $i\gg 0$) for $\ov{\mf{g}}=\ov{\mf{a}}$ (resp. $\ov{\mf{c}},\ov{\mf{d}}$). It is easy to see that $d\in\R$ and $\xi(\wt{E}_{ii})-\xi(\wt{E}_{i+1,i+1})\in\Z_+$ (resp. $\xi(E_{ii})-\xi(E_{i+1,i+1})\in\Z_+$) for all $i$ (resp. $i\not=0$) for $\ov{\mf{g}}=\ov{\mf{c}},\ov{\mf{d}}$ (resp. $\ov{\mf{a}}$). This implies $\xi=\SLa^\g(\la, d)$ for some partition $\lambda$ (resp. pair of partitions $\la=(\la^+,\la^-)$) and $d\in\R$ for $\ov{\mf{g}}=\ov{\mf{c}},\ov{\mf{d}}$ (resp. $\ov{\mf{a}}$). Now applying truncation functor to $L(\ov{\mf{g}}_\infty,\xi)$ with $n\gg d$ (resp. $m, n\gg d$) for $\ov{\mf{g}}=\ov{\mf{c}},\ov{\mf{d}}$ (resp. $\ov{\mf{a}}$),   $\mf{tr}_{\mf{t}\ov\h}(L(\ov{\mf{g}}_\infty,\xi))$ is a unitarizable ${\mf{t}\ov{\g}}$-module with respect to $\omega$. By \thmref{EHW} and \eqnref{trf},  we have $d\in \Z$ and $(\la, d)\in \mc{D}_{\mf{t}}(\g)$. Hence $\xi=\SLa^\g(\la, d)$ for some $(\la, d)\in \mc{D}(\g)$. Conversely, the irreducible highest weight $\ov{\g}_\infty$-modules $L(\ov{\mf{g}}_\infty,\SLa^\g(\la, d))$ are modules appearing in the Howe dualities at negative levels described in \cite{W}. These modules are
unitarizable and quasi-finite. The proof is completed.
\end{proof}

\begin{rem}\label{rem:unitary}
The modules described in the proposition are modules appearing in the Howe dualities at negative levels described in \cite{W} (cf. \cite[Theorem 5.6, 5.8, 5.9]{LZ}).
\end{rem}

\section{Numerical data of the highest weights}\label{hw-data}
In this section, we shall provide combinatorial descriptions of $\SLa^\g(\lambda,d)$ in terms of ${\Lambda}^\g(\lambda,d)$.

\begin{definition} Let $\{a_i\}_{i\in\N}$ and $\{b_i\}_{i\in\N}$ be two strictly
decreasing sequences of integers (resp. half integers). Then the sequences $\{a_i\}_{i\in\N}$ and $\{b_i\}_{i\in\N}$ are said to form a dual pair if $\Z$ (resp. $\frac{1}{2}+\Z$) is the disjoint
union of the two sequences $\{a_i\}_{i\in\N}$ and $\{-b_i\}_{i\in\N}$.
\end{definition}

Define the function $\rho$ on $\N$ by $\rho(i)=-i$ for all $i\in\N$.
The following lemma is well known (see e.g.~\cite[(1.7)]{M}).

\begin{lem}\label{dual pair} For any partition $\la$, the sequences $\{\la_i+\rho(i)\}_{i\in\N}$ and $\{\la'_i+\rho(i)+1\}_{i\in\N}$ form a dual pair.
\end{lem}

Recall that $\phi_{\g}=\sum_{i\in \N}\epsilon_i$ for $\g=\mathfrak{c,d}$.

\begin{lem}\label{zeta-bcd}
For $\g=\mf{c,d}$ and $(\la, d)\in \mc{D}(\g)$, let $\{\zeta_i\}_{i\in\N}$ and $\{\ov\zeta_i\}_{i\in\N}$ be two sequences determined by
\begin{align*}
\Lambda^\g(\la,d)+\rho_\g-d\langle\vartheta_\g,K\rangle\phi_\g&=\sum_{i\in I_\g}\zeta_i\epsilon_i +d\vartheta_\g,\\ \SLa^\g(\lambda,d)+\rho_{\ov\g}+d\langle\vartheta_\g,K\rangle\phi_{\ov\g}
&=\sum_{i\in I_\g}\ov\zeta_i\epsilon_i -\frac{d\langle\vartheta_\g,K\rangle}{\langle\vartheta_{\ov{\g}},K\rangle}\vartheta_{\ov{\g}}.
\end{align*}
Then $\{\zeta_i\}_{i\in\N}$ and $\{\ov\zeta_i\}_{i\in\N}$ form a dual pair. Moreover, $\zeta_i<0$ for $i\in\N$ and $\g\not=\mf{d}$.  In the case $\g=\mf{d}$, $\zeta_i<0$ for $i\ge 2$, and $\zeta_1<0$ (resp. $=0$ and $>0$) for $\la'_1< \frac{d}{2}$ (resp. $=\frac{d}{2}$ and $>\frac{d}{2}$).
\end{lem}

\begin{proof} By \lemref{dual pair}, $\{\zeta_i\}_{i\in\N}$ and $\{\ov\zeta_i\}_{i\in\N}$ form a dual pair. It is clear that $\zeta_i < 0$ for $i\in\N$ and $\g=\mf{c}$. For $\g=\mf{d}$, we have $\la'_2\le \frac{d}{2}$ and hence $\zeta_i < 0$ for $i\ge 2$. Also, $\zeta_1=\la'_1-\frac{d}{2}<0$ (resp. $=0$ and $>0$) for $\la'_1< \frac{d}{2}$ (resp. $=\frac{d}{2}$ and $>\frac{d}{2}$).
\end{proof}

\begin{lem}\label{S-bcd}
For $\g=\mf{c,d}$ and $(\la, d)\in \mc{D}(\g)$, let $\{\zeta_i\}_{i\in\N}$ and $\{\ov\zeta_i\}_{i\in\N}$ be two sequences defined in \lemref{zeta-bcd}. Define $N(\la, d)=\{(i,j)\in \N\times \N\,\mid\,\ov\zeta_i+\ov\zeta_j=0,\,i, j\in\N\}$, ${J}=\{j\in\N\,\mid\,(j,k)\notin N(\la, d),\,\forall k\in \N\}$, $\mc{S}=\{\zeta_i\,\mid\,i\ge 1\}$ and $\ov{\mc{S}}=\{\ov\zeta_i\,\mid\,i\in{J}\,\}$.
\begin{itemize}
\item[i.] For $\g=\mf{c}$, we have $\ov{\mc{S}}=\mc{S}$ and $\ov\zeta_{d+1}=0$.
\item[ii.] For $\g=\mf{d}$, we have
\begin{eqnarray*}
& \ov{\mc{S}}=\mc{S} \text{ and } \zeta_i\not=0\not=\ov\zeta_i\,\text{for all }\,i,\quad &{\rm if} \,\, \text{$d$ is odd};\\
& \ov{\mc{S}}\cup\{0\}=\mc{S}\text{ and } \zeta_1=0,\quad &{\rm if} \,\,\text{$d$ is even and }\, \,\la'_1=\frac{d}{2};\\
& \ov{\mc{S}}=\mc{S}\text{ and } \ov\zeta_i=0\,\text{for some }\, i,\quad &{\rm if} \,\,\text{$d$ is even and }\, \,\la'_1\not=\frac{d}{2}.
\end{eqnarray*}
\end{itemize}
\end{lem}

\begin{proof} We shall only prove the case $\g=\mf{d}$. The proof of the other cases are similar and easier. For $j\ge 2$, we have $\zeta_1+\zeta_j\le \la'_1+\la'_2-d-j+1\le -1$ and hence $\zeta_1\not=-\zeta_j$ for $j\ge 2$. Since $\{\zeta_i\}_{i\in\N}$ and $\{\ov\zeta_i\}_{i\in\N}$ form a dual pair, $\zeta_1\not=\pm\zeta_j$ for $j\ge 2$ and $\zeta_i$ are negative for $i\ge 2$, we have $\zeta_i\in \ov{\mc{S}}$ for $i\ge 2$, and $\zeta_1\in \ov{\mc{S}}$ for $\zeta_1\not=0$. This implies $\ov{\mc{S}}\supseteq\mc{S}\backslash\{0\}$. For $x\in \ov{\mc{S}}$, we have $-x\notin \ov{\mc{S}}$ and hence $-x\in -\mc{S}$. Therefore $\ov{\mc{S}}=\mc{S}\backslash\{0\}$. By \lemref{dual pair}, $\mc{S}$ (resp. $\ov{\mc{S}}$) contains $0$ if and only if $d$ is even and $\la'_1=\frac{d}{2}$ (resp. $\la'_1\not=\frac{d}{2}$). The proof is completed.
\end{proof}

Recall that $\phi_{\mf{a}}=-\sum_{i\le 0}\epsilon_i$.

\begin{lem}\label{A}
For $(\la, d)\in \mc{D}(\mf{a})$, let $\{\zeta_i\}_{i\in\Z}$ and $\{\ov\zeta_i\}_{i\in\Z}$ be two sequences determined by
\begin{align*}
\Lambda^\mf{a}(\la,d)+\rho_\mf{a}-d\phi_\mf{a}&=\sum_{i\in \Z}(\zeta_i-1)\epsilon_i +d\vartheta_\mf{a},\\
\SLa^\mf{a}(\lambda,d)+\rho_\mf{a}+d\phi_\mf{a}
&=\sum_{i\in\Z}\ov\zeta_i\epsilon_i -d\vartheta_\mf{a}.
\end{align*}
Define $N(\la,d)=\{(i,j)\in I_\mf{a}\times I_\mf{a}\,\mid\,\ov\zeta_i=\ov\zeta_j,\,i\le 0<j\}$,
${J}_+=\{j\in\N\,\mid\,(i,j)\notin N(\la, d),\forall i\le 0\}$, ${J}_-=\{i\in\Z\,\mid\,(i,j)\notin N(\la, d),\forall j\in \N\}$, $\mc{S}_+=\{\zeta_i\,\mid\,i\ge 1\}$, $\mc{S}_-=\{\zeta_i\,\mid\,i\le 0\}$, $\ov{\mc{S}}_+=\{\ov\zeta_i\,\mid\,i\in {J}_+\,\}$ and $\ov{\mc{S}}_-=\{\ov\zeta_i\,\mid\,i\in {J}_-\,\}$.
Then we have $\ov{\mc{S}}_+=-\mc{S}_-$ and $\ov{\mc{S}}_-=-\mc{S}_+$.
\end{lem}

\begin{proof} Let $\mc{B}_+=\{\ov\zeta_i\,\mid\,i\in\N\}$ and $\mc{B}_-=\{\ov\zeta_i\,\mid\,i\le 0\}$. By \lemref{dual pair}, we have
 \[(-\mc{S}_+)\sqcup \mc{B}_+=\Z\quad {\rm and} \quad(-\mc{S}_-)\sqcup \mc{B}_-=\Z.\]
For $x\in\ov{\mc{S}}_+$, we have $x\notin\mc{B}_-$ by the definition of $\ov{\mc{S}}_+$ and hence $x\in-\mc{S}_-$. Therefore $\ov{\mc{S}}_+\subseteq-\mc{S}_-$.
Now assume $x\in-\mc{S}_-$. We have $x\notin\mc{B}_-$. Since $\{\zeta_i\}_{i\in\Z}$ is strictly increasing, we have $x\notin-\mc{S}_+$ and hence $x\in\mc{B}_+$. Thus $x\in\mc{B}_+\backslash\mc{B}_-=\ov{\mc{S}}_+$ and therefore $-\mc{S}_-\subseteq\ov{\mc{S}}_+$. Similarly, we have $-\mc{S}_+=\ov{\mc{S}}_-$. The proof is completed.
\end{proof}

We shall use the notations defined in \lemref{S-bcd} and \lemref{A} for the rest of the paper.
By \eqnref{action-a} and \lemref{A}, we have (for $(\la,d)\in \mc{D}(\mf{a})$, $\sigma\in W_{\mf{a}}$)
\begin{equation}\label{sigma-a}
\sigma^{-1}(\Lambda^\mf{a}(\la,d)+\rho_\mf{a})
=\sum_{i\in\Z}(\zeta_i-1)\epsilon_{\sigma^{-1}(i)}+d\phi_\mf{a}+d\vartheta_\mf{a}
=\sum_{i\in\Z}\zeta_{\sigma(i)}\epsilon_i-\sum_{i\in\Z}\epsilon_i+d\phi_\mf{a}+d\vartheta_\mf{a}.
\end{equation}
By \lemref{S-bcd} and \eqnref{action-bcd}, we have (for $(\la,d)\in \mc{D}(\g)$, $\sigma\in W_{\g}$ and $\g=\mf{c,d}$)
\begin{equation}\label{sigma-bcd}
\sigma^{-1}(\Lambda^\g(\la,d)+\rho_\g)=\sum_{i\in \N} \zeta_{\sigma(i)}\epsilon_{i}+d\langle\vartheta_\g,K\rangle\phi_\g +d\vartheta_\g+d\vartheta_\g.
\end{equation}

For $\eta$ belonging to the subspace of $\mathfrak{h}_{\ov\g}^\ast$ spanned by  $\epsilon_j$s and $\vartheta_{\ov\g}$, let
$[\eta]^+$ denote the unique $\Delta^+_{\ov\g,c}$-dominant element in
${W}_{\ov\g,0}$-orbit of $\eta\in \h_{\ov\g}^*$. The following two propositions are  important for proving the main theorem in the next section.

\begin{prop}\label{hw-bcd} Let $\{j_i\}_{i\in\N}$ be the strictly increasing sequence with $J=\{j_i\mid\,i\in\N\}$. For $(\la, d)\in\mc{D}(\g)$ with $\g=\mf{c,d}$ and a partition $\mu$ with $\Lambda^\g(\mu,d)=\sigma^{-1}\circ \Lambda^\g(\la,d)$ for some $\sigma\in
{W}^0_{\g,k}$,  we have \begin{align*}
&\SLa^\g(\mu,d)
+\rho_{\ov\g}+d\langle\vartheta_\g,K\rangle\phi_{\ov\g}\\
=&\begin{cases}
\big[\sum_{i\in \N\backslash J}\ov\zeta_i\epsilon_i +\sum_{i\in \N}\ov\zeta_{j_{\sigma(i)}}\epsilon_{j_i} -\frac{d\langle\vartheta_\g,K\rangle}{\langle\vartheta_{\ov{\g}},K\rangle}\vartheta_{\ov{\g}}\big]^+,
\quad \text{if } 0\notin\mc{S};\\
\big[\sum_{i\in \N\backslash J}\ov\zeta_i\epsilon_i +\sum_{i\in \N}\ov\zeta_{j_{\sigma^0(i)}}\epsilon_{j_i} -\frac{d\langle\vartheta_\g,K\rangle}{\langle\vartheta_{\ov{\g}},K\rangle}\vartheta_{\ov{\g}}\big]^+,
\quad \text{if } 0\in\mc{S}.
\end{cases}
\end{align*}
Here $\sigma^0$ appears only in the case $\g=\mf{d}$  and it is determined by
$\ov{\sigma^0}=\sigma$ (see \lemref{sigma-bar} and \lemref{c-d}).
\end{prop}

\begin{proof} In the proof, union means disjoint union. Let $\{\xi_i\}_{i\in\N}$ and $\{\ov\xi_i\}_{i\in\N}$ be two sequences determined by
\begin{align*}
\Lambda^\g(\mu,d)+\rho_\g-d\langle\vartheta_\g,K\rangle\phi_\g&=\sum_{i\in I_\g}\xi_i\epsilon_i +d\vartheta_\g,\\ \SLa^\g(\mu,d)+\rho_{\ov\g}+d\langle\vartheta_\g,K\rangle\phi_{\ov\g}
&=\sum_{i\in I_\g}\ov\xi_i\epsilon_i -\frac{d\langle\vartheta_\g,K\rangle}{\langle\vartheta_{\ov{\g}},K\rangle}\vartheta_{\ov{\g}}.
\end{align*}
Assume $0\notin\mc{S}$. We have $\{\ov\zeta_{j_i}\}_{i\in\N}=\{\zeta_i\}_{i\in\N}$ by \lemref{S-bcd}. By  \lemref{zeta-bcd}, \lemref{S-bcd}, and the fact that $\sigma$ acts on $\Z^*$ as a signed permutation, we have
\[\{-\zeta_{\sigma(i)}\,\mid i\in\N\}\sqcup\{\ov\zeta_{j_{\sigma(i)}}\,\mid i\in\N\} \sqcup\{\ov\zeta_i\,\mid i\in\N\backslash J\}=\Z\,\,({\rm or}\,\,\hf+\Z).
\]
Since $\{\xi_i\}_{i\in\N}$ and $\{\ov\xi_i\}_{i\in\N}$ form a dual pair, and $\{\xi_i\mid\,i\in\N\}=\{\zeta_{\sigma(i)}\,\mid i\in\N\}$ by \eqnref{sigma-bcd}, we have $\{\ov\xi_i\,\mid i\in\N\}=\{\ov\zeta_{j_{\sigma(i)}}\,\mid i\in\N\} \sqcup\{\ov\zeta_i\,\mid i\in\N\backslash J\}$. Therefore the proposition holds for this case since $\{\ov\xi_i\}_{i\in\N}$ is a decreasing sequence.

The case of $0\in\mc{S}$ only occurs when $\g=\mf{d}$ with $\zeta_1=0$. We have $\{\ov\zeta_{j_i}\mid i\in\N\}=\{\zeta_i\mid i\in\N\}\backslash\{0\}$. Since $\sigma$ acts on $\Z^*$ as a signed permutation, by \lemref{c-d}, we have
\[\{-\zeta_{\sigma(i)}\,\mid i\in\N\}\sqcup\{\ov\zeta_{j_{\sigma^0(i)}}\,\mid i\in\N\} \sqcup\{\ov\zeta_i\,\mid i\in\N\backslash J\}=\Z.
\]
Now the proposition also follows in this case using the arguments above.
\end{proof}

\begin{prop}\label{hw-a} Let $\{j_i\}_{i\in\Z}$ be the strictly decreasing sequence such that $J_+=\{j_i\mid\,i\le 0\}$ and $J_-=\{k_i\mid\,i\in\N\}$, and let $J=J_-\sqcup J_+$. For $(\la, d)\in\mc{D}(\mf{a})$ and a partition $\mu$ with $\Lambda^\mf{a}(\mu,d)=\sigma^{-1}\circ \Lambda^\mf{a}(\la,d)$ for some $\sigma\in
{W}^0_{\mf{a},k}$, we have \[\SLa^\mf{a}(\mu,d)
+\rho_{{\mf{a}}}+d\phi_{{\mf{a}}}=
\big[\sum_{i\in \Z\backslash J}\ov\zeta_i\epsilon_i +\sum_{i\in \Z}\ov\zeta_{j_{\sigma(i)}}\epsilon_{j_i} -d\vartheta_{{\mf{a}}}\big]^+.
\]
\end{prop}

\begin{proof} In the proof, union means disjoint union. Let $\{\xi_i\}_{i\in\Z}$ and $\{\ov\xi_i\}_{i\in\Z}$ be two sequences determined by
\begin{align*}
\Lambda^\mf{a}(\mu,d)+\rho_\mf{a}+\sum_{i\in \Z}\epsilon_i -d\phi_\mf{a}&=\sum_{i\in \Z}\xi_i\epsilon_i +d\vartheta_\mf{a},\\ \SLa^\mf{a}(\mu,d)+\rho_{{\mf{a}}}+d\phi_{{\mf{a}}}
&=\sum_{i\in \Z}\ov\xi_i\epsilon_i -d\vartheta_{{\mf{a}}}.
\end{align*}
By \lemref{A}, we have
\begin{equation*}\Z=(-\mc{S}_+)\sqcup(\ov{\mc{S}}_+)\sqcup\{\ov\zeta_i\,\mid i\in\N\backslash J_+\}=(-\mc{S}_+)\sqcup(-\mc{S}_-)\sqcup\{\ov\zeta_i\,\mid i\in\N\backslash J_+\}.
\end{equation*}
Therefore we have $\Z=\{-\zeta_{\sigma(i)}\,\mid i\in\Z\}\sqcup\{\ov\zeta_i\,\mid i\in\N\backslash J_+\}$ because $\sigma$ acts as a permutation on $\Z$. Since $\xi_i=\zeta_{\sigma(i)}$ for $i\in \Z$ by \eqnref{sigma-a} and $\zeta_{\sigma(i)}=-\ov\zeta_{j_{\sigma(i)}}$ for $i\in \Z$ by \lemref{A}, we have
\begin{align*}\Z&=\{-\zeta_{\sigma(i)}\,\mid  i\in\N\}\sqcup \{-\zeta_{\sigma(i)}\,\mid i\le 0\}\sqcup\{\ov\zeta_i\,\mid i\in\N\backslash J_+\}\\
&=\{-\xi_{i}\,\mid i\in\N\}\sqcup \{\ov\zeta_{j_{\sigma(i)}}\,\mid i\le 0\}\sqcup\{\ov\zeta_i\,\mid i\in\N\backslash J_+\}.
\end{align*}
Since $\{\xi_i\}_{i\in\N}$ and $\{\ov\xi_i\}_{i\in\N}$ form a dual pair, we have $\{\ov\xi_i\,\mid i\in\N\}=\{\ov\zeta_{j_{\sigma(i)}}\,\mid i\le 0\} \sqcup\{\ov\zeta_i\,\mid i\in\N\backslash J_+\}$. Similarly, we have $\{\ov\xi_i\,\mid i\le 0\}=\{\ov\zeta_{j_{\sigma(i)}}\,\mid i\in\N\} \sqcup\{\ov\zeta_i\,\mid i\in(-\Z_+)\backslash J_-\}$. Therefore the proposition holds since $\{\ov\xi_i\}_{i\in\N}$ is a decreasing sequence and $\{\ov\xi_{-i}\}_{i\in\Z_+}$ is an increasing sequence.
\end{proof}

\section{$\mathfrak u_{\ov\g}^-$-homology formulas for $\ov{\mf g}_\infty$-modules}\label{sec:main}
In this section we give a combinatorial proof of Enright's $\mathfrak u_{\ov\g}^-$-homology formula \cite{E} for the unitarizable highest weight $\ov{\mf g}_\infty$-modules with highest weight $\SLa^\g(\la,d)$.

For a module $V$ over Lie algebra $\mc{G}$, let ${\rm H}_k(\mc{G};V)$ denote $k$-th homology group of $\mc{G}$ with coefficients in $V$.
It is well known that the homology groups ${\rm H}_k(\mathfrak{u}_{\ov\g}^-;V)$ are $\mf{l}_{\ov\g}$-modules. The $\mathfrak u_{\ov\g}^-$-homology of unitarizable highest weight modules are described by the following theorem which was obtained in \cite[Theorem 7.2]{CK} for $\ov{\g}_\infty=\mf{a}_\infty$ and in \cite[Theorem 6.5]{CKW} for $\ov{\g}_\infty=\mf{c}_\infty, \mf{d}_\infty$. The following theorem holds for more general situation by using the correspondence of homology group in the sense of super duality \cite[Theorem 4.10]{CLW} together with Kostant's formulas for integrable $\g_\infty$-modules (cf. \cite{J, Ko,V, CK}).

\begin{thm}\label{homology}
For $(\la,d)\in\mc{D}(\g)$ with $\g=\mf{c,d}$ (resp. $\g=\mf{a}$), we have, as
$\mf{l}_{\ov\g}$-modules,
\begin{equation*}
{\rm H}_k\big(\mf {u}_{\ov\g}^-;L(\ov\g_\infty,\SLa^\g(\la,d))\big)
\cong\bigoplus_{\mu}L(\mf{l}_{\ov\g}\,,\ov{\Lambda}^\g(\mu,d)),
\end{equation*}
where $\mu$ is summed over all partitions (resp. pairs of partitions)
such that $\Lambda^\g(\mu,d)=w^{-1}\circ \Lambda^\g(\la,d)$ for some $w\in
{W}^0_{\g,k}$.
\end{thm}

For $\xi$ belonging the subspace of $\mathfrak{h}_{\ov\g}^\ast$ spanned by $\epsilon_j$s and $\vartheta_{\ov\g}$, let $\Psi(\xi)=\{\alpha\in
\Delta_{\ov\g,n}^+\,|\,(\xi+\rho_{\ov\g}\,\vert\,\alpha)=0\}$ and define $\Phi(\xi)$ to be the
subset of $\Delta_{\ov\g,n}^+$ consisting of roots $\beta$ satisfying the following conditions \cite{E,DES}:
\begin{itemize}
\item[i.] $\langle \xi+\rho_{\ov\g},
\beta^\vee\rangle\in\N$;
\item[ii.] $(\beta\,\vert\,\alpha)=0$ for all $\alpha\in\Psi(\xi)$;
\item[iii.] $\beta$ is short if $\Psi(\xi)$ contains a long root.
\end{itemize}
Let ${W}_{\ov\g}(\xi)$ be the
subgroup of ${W}_{\ov\g}$ that is generated by the reflections $s_\beta$
with $\beta\in\Phi(\xi)$.
Define $\Delta_{\ov\g}(\xi)$ to be the subset of $\Delta_{\ov\g}$ consisting of the roots
$\vartheta\in\Delta_{\ov\g}$ such that $s_\vartheta$ lies in ${W}_{\ov\g}(\xi)$.

For $(\la,d)\in \mc{D}(\g)$, let $\Delta_{\ov\g}(\la,d)=\Delta_{\ov\g}(\ov{\La}^\g(\la,d))$
and ${W}_{\ov\g}(\la,d)={W}_{\ov\g}(\ov{\La}^\g(\la,d))$. Then $\Delta_{\ov\g}(\la,d)$ is an abstract root system
and ${W}_{\ov\g}(\la,d)$ is the Weyl group of $\Delta_{\ov\g}(\la,d)$ \cite{E, EHW} (see also \lemref{D} below). Let $\Delta_{\ov\g}^+(\la,d)=\Delta_{\ov\g}(\la,d)\cap\Delta_{\ov\g}^+$ be the set of positive roots of $\Delta_{\ov\g}(\la,d)$.
Set ${W}_{\ov\g,0}(\la,d)={W}_{\ov\g}(\la,d)\cap {W}_{\ov\g,0}$. Let ${W}_{\ov\g}^0(\la,d)$ denote the set of the
minimal length representatives of the left coset space ${W}_{\ov\g}(\la,d)/{W}_{\ov\g,0}(\la,d)$ and let ${W}_{\ov\g,k}^0(\la,d)$ be the subset of ${W}_{\ov\g}^0(\la,d)$ consisting of elements $\sigma$ with $\ell_{(\la,d)}(\sigma)=k$, where $\ell_{(\la,d)}$ is the length
function on $W_{\ov\g}(\la,d)$.

For $(\la,d)\in \mc{D}(\g)$, let ${J}^0=J\sqcup \{j\mid \ov{\zeta}_j=0\}$ for $\g=\mf{c,d}$ and define
\begin{align*}
\Upsilon(\la,d)=
\begin{cases}
\{\epsilon_i-\epsilon_j \in\Delta_{\ov\g}^+;\, ( i\in J_-, \, j\in J_+)\}, \,\, {\rm for}\,\, \g=\mf{a};\\
\{-\epsilon_i-\epsilon_j\in\Delta_{\ov\g}^+;\, ( i<j,\, i, j\in J^0)\},  \,\, {\rm for} \,\, \g=\mf{c};\\
\{-\epsilon_i-\epsilon_j\in\Delta_{\ov\g}^+;\, ( i<j,\, i, j\in J)\},  \,\, {\rm if}\,\, J^0\not=J\,\, {\rm or}\,\, \frac{d}{2}\notin\Z, {\rm for} \,\, \g=\mf{d};\\
 \{-\epsilon_i-\epsilon_j,-2\epsilon_i\in\Delta_{\ov\g}^+;\, ( i<j,\, i, j\in J)\},   \,\, {\rm if}\,\, J^0=J\,\, {\rm and}\,\, \frac{d}{2}\in\Z, {\rm for} \,\, \g=\mf{d}.
\end{cases}
\end{align*}

\begin{lem}\label{D} For $(\la,d)\in \mc{D}(\g)$, we have
\begin{align*}
\Delta_{\ov\g}(\la,d)=
\begin{cases}
\{\epsilon_i-\epsilon_j \in\Delta_{\ov\g};\, (i\not=j,\,\, i, j\in J_-\sqcup J_+)\}, \,\, {\rm for}\,\, \g=\mf{a};\\
\{\pm(\pm\epsilon_i-\epsilon_j)\in\Delta_{\ov\g};\, ( i<j,\, i, j\in J^0)\},  \,\, {\rm for} \,\, \g=\mf{c};\\
\{\pm(\pm\epsilon_i-\epsilon_j)\in\Delta_{\ov\g};\, ( i<j,\, i, j\in J)\},  \,\, {\rm if}\,\, J^0\not=J\,\, {\rm or}\,\, \frac{d}{2}\notin\Z, {\rm for} \,\, \g=\mf{d};\\
 \{\pm(\pm\epsilon_i-\epsilon_j),\pm2\epsilon_i\in\Delta_{\ov\g};\, ( i<j,\, i, j\in J)\},\,{\rm if}\,J^0=J\,{\rm and}\,\frac{d}{2}\in\Z, {\rm for} \,\, \g=\mf{d}.
\end{cases}
\end{align*}
\end{lem}

\begin{proof}For $(\la,d)\in \mc{D}(\g)$, we have  $\Phi(\ov{\La}^\g(\la,d))\subseteq\Upsilon(\la,d)$ by \lemref{S-bcd} and \lemref{A}. Using the relations of the Weyl groups, it is easy to observe that $\Upsilon(\la,d)\subseteq\Delta_{\ov\g}(\la,d)$. Now the lemma follows by using the relations of the Weyl groups again.

\end{proof}
\begin{lem}\label{W-W} For $(\la,d)\in \mc{D}(\g)$, there is a bijection from $W_{\g,k}^0$ to $W_{\ov\g,k}^0(\la,d)$.
\end{lem}
\begin{proof} By \lemref{D}, it is clear that $W_{\g,k}^0=W_{\ov\g,k}^0(\la,d)$ for the cases $\g=\mf{a}$
and $\g=\mf{d}$ with $J^0\not=J$ or $\frac{d}{2}\notin\Z$. For the cases $\g=\mf{c}$ and $\g=\mf{d}$ with $J^0=J$ and $\frac{d}{2}\in\Z$, the lemma follows from \lemref{D} and \lemref{sigma-bar}.
\end{proof}

Using \thmref{homology}, \propref{hw-bcd}, \propref{hw-a}, \lemref{W-W} and \lemref{c-d} together with \eqnref{action-a} and \eqnref{action-bcd}, we have the following theorem.

\begin{thm}\label{main}
For $(\la,d)\in\mc{D}(\g)$ and $k\in \Z_+$, we have, as
$\mf{l}_{\ov\g}$-modules,
\begin{equation*}
{\rm H}_k\big(\mf {u}_{\ov\g}^-;L(\ov\g_\infty,\SLa^\g(\la,d))\big)
\cong\bigoplus_{w\in {W}_{\ov\g,k}^0(\la,d)}L\big(\mf{l}_{\ov\g}\,,[w^{-1}( \SLa^\g(\la,d)+\rho_{\ov\g})]^+ -\rho_{\ov\g}\big).
\end{equation*}
\end{thm}

\begin{rem}\label{cohomology} There has the counterpart of the Theorem 4.11 in \cite{CLW} for $\mf {u}^+$-cohomology in the sense of \cite[Section 4]{L}. The analogous statement is also true for $\g=\mf{a}$.  The formulas for $\mf {u}^+$-cohomology can be proved by the same argument as in the proof in \cite{CLW}.
Therefore, there is an analogue of \thmref{main} for $\mf {u}_{\ov\g}^+$-cohomology in the sense of \cite{L}. The formulas for the cohomology can be proved by the same argument as in the proof of the theorem above.
\end{rem}

\section{Homology formulas for unitarizable modules over finite dimensional Lie algebras}\label{application}

In the section we shall give a new proof of Enright's homology formulas for unitarizable modules over classical Lie algebras corresponding to the three Hermitian symmetric pairs of classical types $(SU(p,q),SU(p)\times SU(q))$, $(Sp(n,\mathbb{R}),U(n))$ and $(SO^\ast(2n),U(n))$.

For $\xi$ belonging to $\h_{\mf{t}\ov\g}^*$, let $\Psi(\xi)=\{\alpha\in
\Delta_{{\mf{t}\ov\g},n}^+\,|\,(\xi+\rho_{{\mf{t}\ov\g}}\,\vert\,\alpha)=0\}$ and define $\Phi(\xi)$ to be the
subset of $\Delta_{{\mf{t}\ov\g},n}^+$ consisting of roots $\beta$ satisfying the following conditions \cite{E,DES}:
\begin{itemize}
\item[i.] $\langle \xi+\rho_{{\mf{t}\ov\g}},\beta^\vee\rangle\in\N$;
\item[ii.] $(\beta\,\vert\,\alpha)=0$ for all $\alpha\in\Psi(\xi)$;
\item[iii.] $\beta$ is short if $\Psi(\xi)$ contains a long root.
\end{itemize}
Let ${W}_{{\mf{t}\ov\g}}(\xi)$ be the
subgroup of ${W}_{{\mf{t}\ov\g}}$ that is generated by the reflections $s_\beta$
with $\beta\in\Phi(\xi)$.
Associated to ${W}_{{\mf{t}\ov\g}}(\xi)$, let $\Delta_{{\mf{t}\ov\g}}(\xi)$ denote the subset of $\Delta_{{\mf{t}\ov\g}}$ consisting of the roots
$\vartheta$ such that $s_\vartheta$ lies in ${W}_{{\mf{t}\ov\g}}(\xi)$. We also let
$[\xi]^+$ be the unique $\Delta^+_{{\mf{t}\ov\g},c}$-dominant element in
the ${W}_{{\mf{t}\ov\g},0}$-orbit of $\xi$.

Assume that the irreducible module $L({\mf{t}\ov{\g}},\xi)$ is unitarizable with highest weight $\xi\in \h_{\mf{t}\ov{\g}}^*$. Then $\Delta_{\mf{t}\ov\g}(\xi)$ is an abstract root system
and ${W}_{\mf{t}\ov\g}(\xi)$ is the Weyl group of $\Delta_{\mf{t}\ov\g}(\xi)$ by \cite{E,EHW}.
Let $\Delta_{\mf{t}\ov\g}^+(\xi)=\Delta_{\mf{t}\ov\g}(\xi)\cap\Delta_{\mf{t}\ov\g}^+$ be the set of positive roots of $\Delta_{\mf{t}\ov\g}(\xi)$. Set ${W}_{\mf{t}\ov\g,0}(\xi)={W}_{\mf{t}\ov\g}(\xi)\cap {W}_{\mf{t}\ov\g,0}$. Let ${W}_{\mf{t}\ov\g}^0(\xi)$ denotes the set of the
minimal length representatives of the left coset space ${W}_{\mf{t}\ov\g}(\xi)/{W}_{\mf{t}\ov\g,0}(\xi)$ and let ${W}_{\mf{t}\ov\g,k}^0(\xi)$ be the subset of ${W}_{\mf{t}\ov\g}^0(\xi)$ consisting of elements $\sigma$ with $\ell_{\xi}(\sigma)=k$, where $\ell_{\xi}$ is the length
function on $W_{\mf{t}\ov\g}(\xi)$.

\begin{thm}\label{Enright} For $\ov{\g}=\mf{a, c}$ or $\mf{d}$, let $L({\mf{t}\ov{\g}},\xi)$ be a unitarizable ${\mf{t}\ov{\g}}$-module with highest weight $\xi\in \h_{\mf{t}\ov{\g}}^*$.
Assume that $\xi$ satisfies the assumption of the Case $(iii)$ of \thmref{EHW} (cf. Case $(ii)$ of \cite[Theorem 9.4]{EHW}) for $\mf{t}\ov\g\cong\mf{so}(2n)$.
For $k\in \Z_+$, we have, as
$\mf{l}_{\mf{t}\ov{\g}}$-modules,
\begin{equation*}
H_k\big(\mf {u}_{\mf{t}\ov{\g}}^-;L({\mf{t}\ov{\g}},\xi)\big)
\cong\bigoplus_{w\in {W}_{{\mf{t}\ov{\g}},k}^0(\xi)}L\big(\mf{l}_{\mf{t}\ov{\g}}\,,[w^{-1}( \xi+\rho_{\mf{t}\ov{\g}})]^+ -\rho_{\mf{t}\ov{\g}}\big).
\end{equation*}
\end{thm}

\begin{proof}  Since  $H_k(\mf {u}_{\mf{t}\ov{\mf{a}}}^-;L({\mf{t}\ov{\mf{a}}},\mu+k\sum_{i=-m+1}^{n}\epsilon_i))=H_k(\mf {u}_{\mf{t}\ov{\mf{a}}}^-;L({\mf{t}\ov{\mf{a}}},\mu))\otimes L(\mf{l}_{\mf{t}\ov{\mf{a}}}, k\sum_{i=-m+1}^{n}\epsilon_i)$ for all $i\ge 0$ and $\mu\in \mf{h}_{\mf{t}\ov{\mf{a}}}^*$, it is sufficient to show all $\xi$ with $k=0$ appearing in the Case $(i)$ of \thmref{EHW} when $\ov{\g}=\ov{\mf{a}}$.

 First we assume that $\xi=\Gamma_{\mf{t}\ov{\g}}(\la,d)$ with $d\notin \Z$. Then we have  $\Delta_{\mf{t}\ov{\g}}(\xi)=\emptyset$ and $L({\mf{t}\ov{\g}},\xi)=N({\mf{t}\ov{\g}},\xi)$ by \thmref{EHW}. Therefore $L({\mf{t}\ov{\g}},\xi)$ is a free $\mf {u}_{\mf{t}\ov{\g}}^-$-modules and hence $H_k\big(\mf {u}_{\mf{t}\ov{\g}}^-;L({\mf{t}\ov{\g}},\xi)\big)=L(\mf{l}_{\mf{t}\ov{\g}}\,,\xi)$ (resp. $=0$) for $k=0$ (resp. $k>0$). Thus the theorem holds for this case.

Now we assume that $\xi=\Gamma_{\mf{t}\ov{\g}}(\la,d)$ for some $(\la,d)\in \mc{D}_\mf{t}(\g)$. By a direct calculation, we have $\Delta_{{\mf{t}\ov\g}}(\xi)=\Delta_{\ov\g}(\la,d)\cap\Delta_{\mf{t}\ov\g}$. Recall that $\mf{tr}_{\mf{t}\ov\h}(L(\ov\g_\infty,\SLa^\g(\la,d))) =L({\mf{t}\ov{\g}},\Gamma_{\mf{t}\ov{\g}}(\la,d))$ for $(\la,d)\in\mc{D}_\mf{t}(\g)$.
Since $\mf{tr}_{\mf{t}\ov\h}(L(\ov\g_\infty,\SLa^\g(\la,d))) =L({\mf{t}\ov{\g}},\Gamma_{\mf{t}\ov{\g}}(\la,d))$ and the homology commutes with the truncation functor, we have
$$H_k(\mf {u}_{\mf{t}\ov{\g}}^-;L({\mf{t}\ov{\g}},\xi))=\mf{tr}_{\mf{t}\ov\h}(H_k(\mf {u}_{\ov\g}^-;L(\ov\g_\infty,\SLa^\g(\la,d)))).$$
%The same results are also true for $\ov{\g}=\mf{a}$ by using the same method in \cite{CLW}.
Note that $H_k(\mf {u}_{\ov\g}^-;L(\ov\g_\infty,\SLa^\g(\la,d)))$ with $k\ge 0$ decompose into the direct sum of irreducible $\mf{l}_{\ov\g}$-modules of the form $L(\mathfrak{l}_{\ov{\g}},\SLa^\g(\mu,d))$ for some partition $\mu$ (resp. pair of partitions $\mu=(\mu^-,\mu^+)$) if $\g=\mf{c,d}$ (resp. $\mf{a}$) and $\mf{tr}_{\mf{t}\ov\h}(L(\mathfrak{l}_{\ov{\g}},\SLa^\g(\mu,d)))
=L(\mathfrak{l}_{\mf{t}\ov{\g}},\Gamma_{\mf{t}\ov{\g}}(\mu,d))$. Therefore, the theorem also holds for this case by \thmref{main}.
\end{proof}

\begin{rem} By \remref{cohomology}, Enright's cohomology formulas for unitarizable modules over classical Lie algebras with highest weights satisfying the assumption in the theorem above can be proved in the same manner as in above.
\end{rem}

\end{document}